\documentclass[12pt]{amsart}
\usepackage{amssymb}
\usepackage[margin=1in]{geometry}

\title{$C^*$-norms for tensor products of discrete group $C^*$-algebras}
\author{Matthew Wiersma}
\address{Department of Pure Mathematics, University of Waterloo, Waterloo, ON, Canada N2L 3G1}
\email{mwiersma@uwaterloo.ca}

\newtheorem{theorem}{Theorem}[section]

\newtheorem{prop}[theorem]{Proposition}
\newtheorem{lemma}[theorem]{Lemma}

\theoremstyle{remark}
\newtheorem{remark}[theorem]{Remark}

\theoremstyle{definition}

\newcommand{\fn}{\!:}
\newcommand{\C}{\mathbb C}

\newcommand{\Hi}{\mathcal{H}}
\newcommand{\lla}{\left\langle}
\newcommand{\rra}{\right\rangle}
\newcommand{\mc}{\mathcal}

\newcommand{\tn}{\textnormal}

\newcommand{\F}{\mathbb F}

\begin{document}
\maketitle

\begin{abstract}
Let $\Gamma$ be a discrete group. We show that if $\Gamma$ is nonamenable, then the algebraic tensor products $C^*_r(\Gamma)\otimes C^*_r(\Gamma)$ and $C^*(\Gamma)\otimes C^*_r(\Gamma)$ do not admit unique $C^*$-norms. Moreover, when $\Gamma_1$ and $\Gamma_2$ are discrete groups containing copies of noncommutative free groups, then $C^*_r(\Gamma_1)\otimes C^*_r(\Gamma_2)$ and $C^*(\Gamma_1)\otimes C_r^*(\Gamma_2)$ admit $2^{\aleph_0}$ $C^*$-norms. Analogues of these results continue to hold when these familiar group $C^*$-algebras are replaced by appropriate intermediate group $C^*$-algebras.
\end{abstract}

\maketitle

\section{Introduction}
Let $\mc A$ and $\mc B$ be $C^*$-algebras. It is always possible to put a $C^*$-norm on the algebraic tensor product $\mc A\otimes \mc B$. For example, the spatial (or minimal) tensor product norm $\|\cdot\|_{\min}$ and the maximal tensor product $\|\cdot\|_{\max}$ are always $C^*$-norms on $\mc A\otimes\mc B$. As the names suggest, the spatial tensor norm is the smallest $C^*$-norm one can place on $\mc A\otimes \mc B$ and the maximal is the largest. In general these norms do not agree. The $C^*$-algebra $\mc A$ is said to be nuclear if $\mc A\otimes \mc B$ admits a unique $C^*$-norm for all choices of $C^*$-algebras $\mc B$ or, equivalently, if $\|\cdot\|_{\max}=\|\cdot\|_{\min}$ on $\mc A\otimes\mc B$ for every $C^*$-algebra $\mc B$.

Let $G$ be a locally compact group. If $G$ is amenable, then $C^*_r(G)$ is a nuclear $C$*-algebra. The converse of this theorem is false in general since, for instance, $C^*_r(G)$ is nuclear for every type I group $G$ \cite{t}. However, Lance showed that this condition characterizes amenability in the case of discrete groups \cite{l}. Since the quotient of a nuclear $C^*$-algebra is nuclear, this also gives the characterization that a discrete group $\Gamma$ is amenable if and only if $C^*(\Gamma)$ is nuclear. Lance's proof, however, does not indicate a $C^*$-algebra $\mc B$ such that $C^*_r(\Gamma)\otimes \mc B$ does not have a unique norm when $\Gamma$ is nonamenable. We show that the algebraic tensor products $C^*_r(\Gamma)\otimes C^*_r(\Gamma)$ and $C^*(\Gamma)\otimes C^*_r(\Gamma)$ do not admit a unique $C^*$-norm when $\Gamma$ is a nonamenable discrete group. The analogue of this result holds when $C^*(\Gamma)$ and $C^*_r(\Gamma)$ are replaced by appropriate intermediate group $C^*$-algebras of $\Gamma$. We do not determine whether $C^*(\Gamma)\otimes C^*(\Gamma)$ admits a unique $C^*$-norm for nonamenable $\Gamma$, but note that finding the solution to this problem for the case when $\Gamma=\F_\infty$ would solve the Connes embedding problem \cite{k}.

Recently Narutaka Ozawa and Gilles Pisier demonstrated pairs of $C^*$-algebras $\mc A$ and $\mc B$ such that $\mc A\otimes \mc B$ admits $2^{\aleph_0}$ distinct $C^*$-norms \cite{op}, including the case when $\mc A=\mc B= \mc B(\Hi)$. Although the paper mainly focuses on von Neumann algebras, Ozawa and Pisier also show that $C^*_r(\F_d)\otimes C^*_r(\F_d)$ admits $2^{\aleph_0}$ distinct $C^*$-norms where $\F_d$ is the noncommutative free group on $d\geq 2$ generators. We generalize this result by showing that $C^*_r(\Gamma_1)\otimes C^*_r(\Gamma_2)$ and $C^*(\Gamma_1)\otimes C^*_r(\Gamma_2)$ admit $2^{\aleph_0}$ distinct $C^*$-norms for every pair of discrete groups $\Gamma_1$ and $\Gamma_2$ containing copies of noncommutative free groups. Again, analogues of these results hold when $C^*(\Gamma_1)$ and $C^*_r(\Gamma_2)$ are replaced by intermediate group $C^*$-algebras of $\Gamma_1$ and $\Gamma_2$. Our approach to finding these $C^*$-norms is different from Ozawa and Pisier's, whose proof relies heavily on the simplicity of $C^*_r(\F_d)$.

We have mentioned the term `appropriate intermediate group $C^*$-algebra' a couple times now. Let $G$ be a locally compact group. By an intermediate group $C^*$-algebra of $G$, we mean a $C^*$-completion of $L^1(G)$ with respect to a norm dominating the reduced norm. This class includes the full and reduced group $C^*$-algebras $C^*_r(G)$ and $C^*(G)$. Nate Brown and Eric Guentner have recently defined a very natural class of intermediate group $C^*$-algebras $C^*_{\ell^p}(\Gamma)$ ($1\leq p<\infty$) for discrete groups $\Gamma$ \cite{bg} with the property that $C^*_{\ell^p}(\F_2)$ are distinct for every $p\geq 2$ \cite{o}. What we are actually able to show is that if $\mc A$ is any intermediate group $C^*$-algebra for a nonamenable discrete group $\Gamma$, then $\mc A\otimes C^*_{\ell^p}(\Gamma)$ does not admit a unique $C^*$-norm, and if $\Gamma_1$ and $\Gamma_2$ are discrete groups containing copies of $\F_2$ and $\mc A$ is an intermediate group $C^*$-algebra of $\Gamma_1$, then $\mc A\otimes C^*_{\ell^p}(\Gamma_2)$ admits $2^{\aleph_0}$ distinct $C^*$-norms.

Throughout this paper we will assume a basic familiarity with $C^*$-tensor products, such as provided in \cite{bo}. The main tools used throughout this paper are the theory of coefficient spaces as developed by Eymard \cite{e} and Arsac \cite{a}, and the recently developed theory of $\ell^p$-representations \cite{bg}. See also \cite{kmq}. The primary purpose of the following section is to familiarize the reader with coefficient spaces as required by this paper, and the section is ended with a characterization of tensor products in terms of these spaces. In the third and final section of the paper we give (remarkably short) proofs of the earlier mentioned results, introducing the theory of $\ell^p$-representations where needed.

\section{Coefficient spaces}

In this section we provide a brief introduction to the theory of coefficient spaces as required by this paper and the results mentioned here will be used implicitly throughout the rest of the paper. The reader should see the papers of Arsac \cite{a} and Eymard \cite{e} for more thorough treatments of this subject. The section ends with a characterization of $C^*$-norms on tensor products in terms of these coefficient spaces.

Let $G$ be a locally compact group. The Fourier-Stieltjes algebra $B(G)$ is defined to be the set of coefficient functions $s\mapsto \pi_{x,y}(s):=\lla \pi(s)x,y\rra$ where $\pi\fn G\to \mc B(\Hi_\pi)$ ranges over the (strongly continuous unitary) representations of $G$ and $x,y$ over $\Hi_\pi$. $B(G)$ is naturally identified with the dual space of $C^*(G)$ via the pairing
$<u,f>\,:=\int u(s)f(s)\,ds$ for $u\in B(G)$ and $f\in L^1(G)$. With the norm $B(G)$ attains from being the dual of $C^*(G)$ and pointwise multiplication, the Fourier-Stieltjes algebra becomes a Banach algebra.

Let $\pi\fn G\to \mc B(\Hi)$ be a fixed representation of $G$. The space $A_\pi$ is defined to be the closed linear span of coefficient functions $\pi_{x,y}$ for $x,y\in\Hi$. Then $A_\pi$ is a translation invariant space (under both left and right translations). Moreover, every norm closed translation invariant subspace of $B(G)$ arises in this way. As a distinguished case, the Fourier algebra $A(G)$ is defined to be the space $A_\lambda$ where $\lambda$ denotes the left regular representation of $G$. Although not obvious, it is a consequence of Fell's absorption principle that the Fourier algebra is an ideal of $B(G)$.

Fix a representation $\pi\fn G\to \mc B(\Hi)$ and let $C^*_\pi$ denote the norm closure of $\pi(L^1(G))$ in $\mc B(\Hi)$. Let $B_\pi$ denote the closure of $A_\pi$ with respect to the weak* topology $\sigma(B(G),C^*(G))$. Then $B_\pi$ can be identified as the dual of $C^*_\pi$ via the pairing $<u,\pi(f)>\,=\int u(s)f(s)\,ds$. If $\sigma$ is another representation of $G$, then $B_\pi\subset B_\sigma$ if and only if $\|\pi(f)\|\leq \|\sigma(f)\|$, i.e., if and only if $\pi$ is weakly contained $\sigma$.

Below we list the properties of these spaces which we will make use of:
\begin{itemize}
\item Suppose $H$ is an open subgroup of $G$. Then we can extend every element $u\in B(H)$ to an element $\dot{u}$ in $B(G)$ by defining $\dot u(s)=0$ for $s\in G\backslash H$. The map $u\mapsto \dot u$ is an isometry.
\item If $H$ is a closed subgroup of $G$ and $\pi$ is a representation of $G$, then $A_\pi|_H=A_{\pi|_H}$ \cite{a}. If $H$ is assumed to be open, then $B_\pi|_H=B_{\pi|_H}$.
\item (Herz's restriction theorem) If $H$ is a closed subgroup of $G$, then $A(G)|_H=A(H)$ \cite{a} (see also \cite{he} and \cite{tt}).
\item If $\sigma$ is an amplification of a representation $\pi$ of $G$, then $A_{\sigma}=A_\pi$ \cite{a}.
\item $A(G)$ is the norm closure of $C_c(G)\cap B(G)$ \cite{e}.
\item $G$ is amenable if and only if $B_\lambda$ contains the constant 1 \cite{hu}.
\item Let $\pi$ and $\sigma$ be representations of $G$. Then $A_{\pi\oplus \sigma}=A_\pi+A_\sigma$ and $B_{\pi\oplus \sigma}=B_\pi+B_\sigma$ \cite{a}.
\end{itemize}

We finish this section with a characterization on $C^*$-norms on tensor products of $C^*$-algebras associated to groups. Recall that the Banach space projective tensor product $X\widehat{\otimes} Y$ of Banach spaces $X$ and $Y$ is the completion of the algebraic tensor product $X\otimes Y$ with respect to the norm $\big\|z\|=\inf\{\sum_{i=1}^n\|x_i\|\|y_i\| : z=\sum_{i=1}^n x_i\otimes y_i\}$.

\begin{prop}\label{characterization}
Let $G_1$ and $G_2$ be locally compact groups and $\pi_1,\pi_2$ be representations of $G_1$ and $G_2$. Then we can identify the $C^*$-norms on $C^*_{\pi_1}(G_1)\otimes C^*_{\pi_2}(G_2)$ with the weak*-closed translation invariant subspaces $B_\sigma$ of $B(G_1\times G_2)$ such that $B_\sigma|_{G_1}=B_{\pi_1}$, $B_\sigma|_{G_2}=B_{\pi_2}$, and $B_\sigma\supset B_{\pi_1\times\pi_2}$. For $f_1,\ldots,f_n\in L^1(G_1)$ and $g_1,\ldots, g_n\in L^1(G_2)$, the norm of $\sum_{i=1}^n \pi_1(f_i)\otimes \pi_2(g_i)$ associated to $B_\sigma$ is given by
$$ \bigg\|\sum_{i=1}^n \pi_1(f_i)\otimes \pi_2(g_i)\bigg\|=\bigg\|\sum_{i=1}^n \sigma(f_i\times g_i)\bigg\|.$$

\end{prop}

\begin{proof}
Note that since $L^1(G_1\times G_2)=L^1(G_1)\widehat{\otimes}L^1(G_2)$, we may consider $*$-representations of $C^*_{\pi_1}(G_1)\otimes C^*_{\pi_2}(G_2)$ as being representations of $G_1\times G_2$. It can be checked that this gives a one-to-one correspondence between $*$-representations of $C^*_{\pi_1}(G_1)\otimes C^*_{\pi_2}(G_2)$ and representations $\sigma$ of $G_1\times G_2$ such that $B_\sigma|_{G_1}\subset B_{\pi_1}$ and $B_\sigma|_{G_2}\subset B_{\pi_2}$. Moreover, this immediately gives that if $\sigma$ is a representation of $G_1\times G_2$ corresponding to a $*$-representation $\widetilde \sigma$ of $C^*_{\pi_1}(G_1)\otimes C^*_{\pi_2}(G_2)$ and $f_1,\ldots,f_n\in L^1(G_1)$, $g_1,\ldots, g_n\in L^1(G_2)$, then
$$ \bigg\|\widetilde\sigma\big(\sum_{i=1}^n \pi_1(f_i)\otimes \pi_2(g_i)\big)\bigg\|=\bigg\|\sum_{i=1}^n \sigma(f_i\times g_i)\bigg\|.$$
Finally, a $*$-representation of $C^*_{\pi_1}(G_1)\otimes C^*_{\pi_2}(G_2)$ corresponding to a representation $\sigma$ of $G_1\times G_2$ separates points of $C^*_{\pi_1}(G_1)\otimes C^*_{\pi_2}(G_2)$ if and only if $\|\sigma(\cdot)\|\geq \|\pi_1\times \pi_2(\cdot)\|$ on $L^1(G_1\times G_2)$ if and only if $B_\sigma\supset B_{\pi_1\times \pi_2}$.
\end{proof}

\section{$C^*$-norms of tensor products}

In this section we show that if $\Gamma$ is a nonamenable discrete group, then $C^*_r(\Gamma)\otimes C^*_r(\Gamma)$ and $C^*(\Gamma)\otimes C^*_r(\Gamma)$ do not admit unique $C^*$-norms. We obtain even stronger results when $\Gamma$ contains a copy of the free group. The key idea behind our (remarkably simple) proofs is to find $B_\sigma$ spaces satisfying the conditions of Proposition \ref{characterization} which differ on the diagonal subgroup $\Delta$ of $\Gamma\times\Gamma$. We begin with the following lemma.

\begin{lemma}
Suppose $\Gamma$ is a discrete group and $\mc S$ is a subset of $B(\Gamma\times\Gamma)$ supported on the diagonal subgroup $\Delta$ of $\Gamma\times\Gamma$. Let $A_{\mc S}$ denote the norm closed translation invariant subspace of $B(\Gamma\times\Gamma)$ generated by $\mc S$. Then $A_{\mc S}|_{\Gamma\times \{e\}}\subset A(\Gamma\times \{e\})$ and $A_{\mc S}|_{\{e\}\times \Gamma}\subset A(\{e\}\times \Gamma)$.
\end{lemma}

\begin{proof}
Let $\varphi\in \mc S$ and fix $s=(s_1,s_2)$, $t=(t_1,t_2)$ in $\Gamma\times\Gamma$. Then $\varphi(s(x_1,e)t)$ is nonzero only if $s_1x_1t_1=s_2t_2$, i.e., only if $x_1=s_1^{-1}s_2t_2t_1^{-1}$. Therefore, the translated element $x\mapsto \varphi(sxt)$ in $B(\Gamma\times\Gamma)$ has finite support when restricted to $\Gamma\times\{e\}$ and, thus, its restriction is an element of $A(\Gamma\times\{e\})$. So $A_\mc S |_{\Gamma\times\{e\}}\subset A(\Gamma\times\{e\})$. Similarly, $A_\mc S |_{\{e\}\times\Gamma}\subset A(\{e\}\times\Gamma)$.
\end{proof}

We are now prepared to prove our first main theorem, i.e., that $C^*_r(\Gamma)\otimes C^*_r(\Gamma)$ and $C^*(\Gamma)\otimes C^*_r(\Gamma)$ do not admit unique $C^*$-norms. We use a single argument conducted in the broader setting of all intermediate group $C^*$-algebras on $\Gamma$ between the reduced and full group $C^*$-algebras.

\begin{theorem}\label{1}
Let $\Gamma$ be a nonamenable discrete group and $\pi$ a representation of $\Gamma$ weakly containing the left regular representation. Then $C^*_\pi(\Gamma)\otimes C^*_r(\Gamma)$ does not admit a unique $C^*$-norm.
\end{theorem}

\begin{proof}
By Proposition \ref{characterization}, it suffices to construct two distinct weak*-closed translation invariant subspaces $B_\sigma$ of $B(G)$ with the prescribed conditions.

We first consider the space $B_{\pi\times\lambda}$ associated to the minimal tensor product $C^*_\pi(\Gamma)\otimes_{\min} C^*_r(\Gamma)$. Notice that on the diagonal subgroup $\Delta$ of $\Gamma\times\Gamma$, the space $A_{\pi\times\lambda}$ restricts to $A_{\pi\otimes\lambda}$. By Fell's absorption principle, $\pi\otimes \lambda$ is unitarily equivalent to an amplification of $\lambda$ and, hence, $A_{\pi\times \lambda}|_\Delta$ is simply the Fourier algebra $A(\Delta)$. Since $\Gamma\cong \Delta$ is nonamenable, this implies that $B_{\pi\times\lambda}|_\Delta$ does not contain the constant 1.

Let $\mc S$ be the set of all elements of $B(\Gamma\times\Gamma)$ supported on $\Delta$. Then the weak*-closure $B_\sigma$ of $A_\mc S+A_{\pi\times\lambda}$ satisfies the conditions of Proposition \ref{characterization} as $B_{\sigma}|_{\Gamma\times \{e\}}=B_\pi+\overline{A_{\mc S}|_{\Gamma\times\{e\}}}^{w^*}=B_\pi$ since $A_{\mc S}|_{\Gamma\times\{e\}}=A(\Gamma\times\{e\})$ and $\pi$ weakly contains $\lambda$. Similarly, $B_{\sigma}|_{\{e\}\times\Gamma}=B_\lambda$. As $B_\sigma|_\Delta$ contains the constant 1, we conclude that $C^*_\pi(\Gamma)\otimes C^*_r(\Gamma)$ does not admit a unique $C^*$-norm.
\end{proof}

We now focus on improving this result in the case when $\Gamma$ contains a copy of a noncommutative free group. Towards this goal, we briefly review the theory of $\ell^p$-representations recently introduced by Brown and Guentner \cite{bg}.

A representation $\pi\fn \Gamma\to \mc B(\Hi)$  is said to be an $\ell^p$-representation of $\Gamma$ if $\Hi$ admits a dense linear subspace $\Hi_0$ such that $\pi_{x,x}\in\ell^p(\Gamma)$ for every $x\in \Hi$. For example $\lambda$ is an $\ell^p$-representation for every $p$ since $c_c(\Gamma)\subset \ell^2(\Gamma)$ clearly satisfies the required conditions. A useful fact to us will be that if $\pi$ is an $\ell^p$-representation of $\Gamma$ and $\sigma$ is an arbitrary representation of $\Gamma$, then $\pi\otimes\sigma$ is an $\ell^p$-representation.

Let $\|\cdot\|_{\ell^p}$ be the $C^*$-norm on $\C[\Gamma]$ defined by
$$ \|x\|_{\ell^p}=\sup\{\|\pi(x)\| : \pi\tn{ is an $\ell^p$-representation of }\Gamma\}.$$
$C^*_{\ell^p}(\Gamma)$ is defined to be the completion of $(\C[\Gamma],\|\cdot\|_{\ell^p})$ and $C^*_{\ell^p}(\Gamma)$ admits a faithful $\ell^p$-representation. Brown and Guentner demonstrated that $C^*_{\ell^p}(\Gamma)$ is simply the reduced group $C$*-algebra $C^*_r(\Gamma)$ when $p\leq 2$ but were able to show that there exists $p>2$ so that $C^*_{\ell^p}(\F_2)$ was neither the reduced nor full group $C$*-algebra. Shortly after, Rui Okayasu demonstrated that the canonical quotient map from $C^*_{\ell^p}(\F_2)$ to $C^*_{\ell^q}(\F_2)$ is not injective for $2\leq q\leq p\leq\infty$ \cite[Corollary 3.2]{o}.

Suppose that $H$ is a subgroup of $\Gamma$. If $\pi$ is an $\ell^p$-representation of $H$, then $\mathrm{Ind}_H^\Gamma\pi$ is an $\ell^p$-representation of $\Gamma$ \cite[Theorem 2.4]{w}. It follows that if $A_{\ell^p}(\Gamma)$ is the closed linear span of coefficients of $\ell^p$-representations of $\Gamma$, then $A_{\ell^p}(\Gamma)|_H=A_{\ell^p}(H)$.

We remark that we may replace $C^*_r(\Gamma)$ with $C^*_{\ell^p}(\Gamma)$ in Theorem \ref{1} by using the fact that if $\sigma$ is an $\ell^p$-representation of $\Gamma$, then $\pi\otimes\sigma$ is an $\ell^p$-representation for all representations $\pi$ of $\Gamma$ in place of Fell's absorption principle. We include this case in the statement of our final theorem.

\begin{theorem}\label{2}
Let $\Gamma_1$ and $\Gamma_2$ be discrete groups containing copies of noncommutative free groups. If $\pi$ is a representation of $\Gamma_1$ weakly containing a copy of the left regular representation, then $C^*_\pi(\Gamma_1)\otimes C^*_{\ell^p}(\Gamma_2)$ admits $2^{\aleph_0}$ distinct $C^*$-norms for every $p\in [2,\infty)$.
\end{theorem}

\begin{proof}
Choose a faithful $\ell^p$-representation $\sigma$ for $C^*_{\ell^p}(\Gamma)$, identify a copy of $\F_2$ inside each of $\Gamma_1$ and $\Gamma_2$. Denote the diagonal subgroup of $\F_2\times\F_2\leq \Gamma_1\times \Gamma_2$ by $\Delta$. Then $(\pi\times \sigma)|_\Delta=\pi|_{\F_2}\otimes\sigma|_{\F_2}$ is an $\ell^p$-representation of $\Delta$ implies that $A_{\pi\times\sigma}|_\Delta\subset A_{\ell^p}(\Delta)$.

For each $q>p$, let $\mc S_q$ be the set of all functions in $A_{\ell^q}(\Gamma_1\times\Gamma_2)$ supported on $\Delta$. Then $A_{\mc S_q}|_{\Delta}= A_{\ell^q}(\Delta)$.
Let $B_{\sigma_q}$ be the weak*-closure of $A_{\mc S_q}+A_{\pi\times\sigma}$. Then, by similar reasoning as in the proof of the previous theorem, $B_\sigma$ satisfies the conditions of Proposition \ref{characterization} and $B_{\sigma_q}|_\Delta=B_{\ell^q}(\Delta)$. Since the $C^*$-norms $\|\cdot\|_{\ell^p}$ on $\ell^1(\Delta)$ are distinct for every $q>p$, we conclude that $B_{\sigma_q}$ are distinct for every $q>p$.
\end{proof}

\begin{remark}
Let $G$ be a locally compact group containing an open normal compact subgroup $K$. This happens, for instance, when $G$ is a totally disconnected SIN group. Let $q\fn G\to G/K$ be the canonical quotient map. If $m_K$ is the normalized Haar measure on $K$, then $\varphi\mapsto m_K*\varphi$ is a contraction on $B(G)$ \cite[Proposition 2.18]{e} mapping $A(G)$ onto $A(G/K)\circ q$ since $A(G)=\overline{B(G)\cap C_c(G)}$ and $\varphi\in C_c(G)$ if and only if $q(\mathrm{supp}\,\varphi)$ is finite. It follows that $m_K*B_{\lambda_G}=B_{\lambda_{G/K}}\circ q$.

Suppose $G$ is nonamenable. Then the analogue of Theorem \ref{1} is true for $G$, i.e., if $\pi$ is any representation of $G$ weakly containing $\lambda_G$, then $C^*_\pi(G)\otimes C^*_r(G)$ does not admit a unique $C^*$-norm. Indeed, denote the diagonal subgroup of $G/K\times G/K$ by $\Delta$. Then, making the appropriate changes to the proof of Theorem \ref{1} and taking $\mc S$ to be the set of all elements in $B(G/K\times G/K)\circ q$ supported on $q^{-1}(\Delta)$ produces a second $B_\sigma$ satisfying the conditions of Proposition \ref{characterization}. Suppose $G_1$ and $G_2$ are two locally compact groups containing open normal compact subgroups $K_1$ and $K_2$ so that $G/K_1$ and $G/K_2$ contain noncommutative free subgroups. Then a similar trick as above shows that $C^*_\pi(G_1)\otimes C^*_r(G_2)$ admits $2^{\aleph_0}$ distinct $C^*$-norms for every representation $\pi$  of $G_1$ weakly containing $\lambda_{G_1}$.
\end{remark}

\section*{Acknowledgements}\label{ackref}
The author would like to thank his supervisor Nico Spronk for suggesting this problem. The author also wishes to thank both his advisors, Brian Forrest and Nico Spronk, for the useful discussions and suggestions.
This research was conducted at the Fields Institute during the thematic program on Abstract Harmonic Analaysis, Banach and Operator Algebras while the author was supported by an NSERC Postgraduate Scholarship.

\end{document}